\documentclass[11pt, a4paper, reqno]{amsart}

\usepackage[textwidth=391pt, marginparwidth=70pt, centering]{geometry}
\usepackage{xcolor}
\usepackage[colorlinks, backref = true]{hyperref} 
\hypersetup{allcolors=blue, pdfstartview=FitH}
\usepackage{graphicx} 
\usepackage{epstopdf} 

\usepackage{soul} 

\usepackage{amsmath, amsfonts, amsthm, amssymb}
\usepackage{mathtools}
\usepackage{enumitem}

\newtheorem{thm}{Theorem}[section]
\newtheorem{corollary}[thm]{Corollary}
\newtheorem{lemma}[thm]{Lemma}

\theoremstyle{remark}

\theoremstyle{definition}

\theoremstyle{definition}
\newtheorem*{ack}{Acknowledgements}

\theoremstyle{definition}
\newtheorem*{out}{Outline}

\newcounter{example}[section]

\numberwithin{equation}{section}
\allowdisplaybreaks

\newcommand{\N}{\mathbb{N}}
\newcommand{\Z}{\mathbb{Z}}
\newcommand{\Q}{\mathbb{Q}}
\newcommand{\R}{\mathbb{R}}
\newcommand{\C}{\mathbb{C}}
\newcommand{\Qtr}{\Q^{\mathrm{tr}}}

\newcommand{\Par}{\mathrm{Par}}
\newcommand{\PCF}{\mathrm{PCF}}
\newcommand{\M}{\mathrm{M}}

\usepackage[
  style = alphabetic,
	backend=biber,
	backref = true,
	url = false,
	doi = true,
	isbn = false,
	url = false,
]{biblatex}
\DefineBibliographyStrings{english}{
  backrefpage  = {$\uparrow$\hspace{-2.5pt}},
  backrefpages = {$\uparrow$\hspace{-2.5pt}},}

\DeclareFieldFormat[article,incollection,inproceedings,book,misc]{title}{\mkbibemph{#1}}

\DeclareFieldFormat{doi}{%
  \mkbibacro{}\addcolon\space%
  \href{https://doi.org/#1}{\textsc{DOI}}%
}

\renewbibmacro{in:}{}

\DeclareFieldFormat[article]{title}{#1}

\DeclareFieldFormat[book]{title}{\mkbibemph{#1}}

\title{Totally real points in the multibrot sets}
\date{\today}

\author[Alessio~Cangini]{Alessio~Cangini}
\address{Alessio Cangini,
    Departement Mathematik und Informatik,
    Universit\"{a}t Basel,
    Spiegelgasse 1,
    4051 Basel,
    Switzerland}
\email{alessio.cangini@unibas.ch}
\urladdr{\url{https://sites.google.com/view/alessio-cangini}}

\author[Hang~Fu]{Hang~Fu}
\address{Hang Fu,
    Departement Mathematik und Informatik,
    Universit\"{a}t Basel,
    Spiegelgasse 1,
    4051 Basel,
    Switzerland}
\email{drfuhang@gmail.com}
\urladdr{\url{https://sites.google.com/view/hangfu}}

\subjclass{Primary: 37P05, 37P15}
\keywords{Arithmetic dynamics, multibrot sets, parabolic points, PCF points}
\thanks{The authors acknowledge support by the Swiss National Science 
Foundation Grant ``Rational points, arithmetic dynamics, and 
heights'' no. 219397.}

\begin{document}

\begin{abstract}
    We classify all totally real parabolic parameters in the multibrot sets, extending a theorem of Buff and Koch.
\end{abstract}
\maketitle
\section{Introduction}\label{section: introduction}

Let~$d \ge 2$ be an integer and consider the family of
polynomials~$f_c(z) = z^d + c \in \C[z]$. For $n \in \N = \{1, 2, \dots\}$, we denote by $f^n_c$ the $n$-th 
iteration of $f_c$.

We focus on three subsets of the parameter space.
The \textit{multibrot set} $\M_d$ consists of all~$c \in \C$ such that the orbit~$\{f_c^n(0) \colon n \in \N\}$ is bounded. When ~$d = 2$, the multibrot set $\M_2$ is the well-known \textit{Mandelbrot~set}. The set of \textit{postcritically finite} parameters, denoted by $\PCF_d$, consists of all $c \in \C$ such that  the orbit~$\{f_c^n(0) \colon n \in \N\}$ is finite.
A point~$z \in \C$ is ~\textit{periodic} for $f_c$ if there is $n \in \N$ such that $f_c^n(z) = z$. If $z \in \C$ is 
periodic and $n \in \N$ is the smallest period, then the \textit{multiplier} of $f_c$ at $z$ 
is given by $(f_c^n)'(z)$. A periodic point is \textit{parabolic} if its multiplier is a root of unity. A parameter $c \in
\C$ is \textit{parabolic} if $f_c$ has a parabolic periodic point; the set of such parameters is denoted by $\Par_d$.

Both $\PCF_d$ and $\Par_d$ are Galois invariant subsets of $\M_d$. An algebraic number is \textit{totally real} if all its Galois conjugates are real. The field of totally real numbers is denoted by $\Qtr$. Noytaptim and Petsche studied the arithmetic properties of $\PCF_d$ and classified $\PCF_d \cap \Qtr$ for all $d \ge 2$.

\begin{thm}[{\cite[Theorem 1]{MR4710201}}]\label{theorem: noytaptim - petsche} We have
\begin{align*}
    \PCF_2 \cap \Qtr & = \{ 0, -1, -2 \}, \\
    \PCF_d \cap \Qtr & = \{ 0, -1 \} \text{ if } d \ge 4 \text{ is even}, \\
    \PCF_d \cap \Qtr & = \{ 0 \}\text{ if } d \ge 3 \text{ is odd}.
\end{align*}
\end{thm}

In order to prove Theorem~\ref{theorem: noytaptim - petsche}, Noytaptim and Petsche used capacity-theoretic tools to bound the degree of points in $\PCF_d \cap \Qtr$. Buff and Koch then realized that, with the help of Kronecker's theorem, they were able to avoid the capacity-theoretic tools and substantially simplify the proof of Theorem~\ref{theorem: noytaptim - petsche} for the case $d=2$. As a further application of their approach, they considered the totally real parabolic points and proved the following result.

\begin{thm}[{\cite[Proposition 2]{buff2022totallyrealpointsmandelbrot}}]\label{theorem: buff - koch}
We have $$\Par_2 \cap \Qtr = \{ {1}/{4}, -{3}/{4}, -{5}/{4}, -{7}/{4}\}.$$
\end{thm}

Given Theorems ~\ref{theorem: noytaptim - petsche} and \ref{theorem: buff - koch}, it is natural to ask if we can also compute the set $\Par_d \cap \Qtr$ for $d \ge 3$. Our main theorem answers this question.

\begin{thm}\label{theorem: theorem}
We have $\Par_3 \cap \Qtr  = \{ \pm {2\sqrt{3}}/{9}\}$ and $\Par_d \cap \Qtr  = \emptyset$ for $d \ge 4$.
\end{thm}

\begin{figure}[h]\label{figure}
    \newpage
    \centering
    \includegraphics[scale = 0.6]{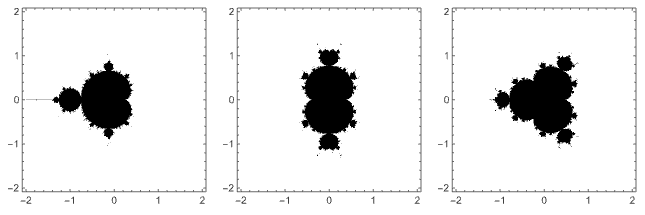}
    \caption{The multibrot sets $\M_2$, $\M_3$ and $\M_4$. The images have been produced with \textsc{Mathematica}.}
\end{figure}
We first note that the simplified proof of Theorem~\ref{theorem: noytaptim - petsche} for the case $d=2$ given by Buff and Koch can be easily generalized to the cases $d\ge 3$ without too much extra effort. However, their proof of Theorem \ref{theorem: buff - koch} relies on the crucial fact that if $c$ is a totally real parabolic parameter, then $4c$ is a totally real algebraic integer, so that Kronecker's theorem can be applied. Due to the lack of this property, their idea does not seem to generalize to the cases $d\ge3$. To tackle this difficulty, we employ the capacity-theoretic tools again.

\begin{out}
In Section \ref{section: background}, we collect some basic properties of $\M_d$ and also review the capacity-theoretic tools mentioned above. In Section \ref{section:3}, we give the proof of Theorem \ref{theorem: noytaptim - petsche}, following the idea of Buff and Koch. In Section \ref{section:4}, we give the proof of Theorem \ref{theorem: theorem}.
\end{out}

\begin{ack} 
We would like to thank Philipp Habegger for the helpful comments on the first draft of this paper.
\end{ack}
\section{Preliminary lemmas}\label{section: background}

In this section, we provide some general lemmas that will be used later. We do not assume any number is totally real.

\subsection{Real slices of multibrot sets}
Since both $\PCF_d \cap \Qtr$ and $\Par_d\cap \Qtr$ are subsets of $\M_d\cap\R$, we first take a closer look at the real slices of multibrot sets.

We define the quantities
    \[
        \alpha(d) = (d-1)d^{-d/(d-1)},\qquad
        \beta(d)  = -2^{1/(d-1)},\qquad
        \gamma(d) = -(d+1)d^{-d/(d-1)},
    \]
where the $(d-1)$-th roots are always taken on the positive real axis. It is known that $\M_d \cap \R = [- \alpha(d), \alpha(d)]$ if $d$ is odd \cite{MR3654403} and $\M_d \cap \R = [\beta(d), \alpha(d)]$ if $d$ is even \cite{ParRanRon}.

When $d$ is odd, for any $c\in\M_d\cap\R$, the polynomial $f_c$ always has a fixed point, either parabolic or attractive. 
When $d$ is even, the dynamics of $f_c$ with $c\in\M_d\cap\R$ is much more complicated. Fortunately, the following lemma enables us to locate the candidate totally real parabolic points in a small subinterval of $\M_d\cap\R$.

\begin{lemma}\label{lemma: lemma 2.1} 
Let $d\geq2$ be an integer. Then:
\begin{itemize}
    \item[(i)] The polynomial $f_{\alpha(d)}$ has a parabolic fixed point.
    \item[(ii)] If $d$ is odd and $-\alpha(d)<c<\alpha(d)$, then $f_{c}$ has an attractive fixed~point.
    \item[(iii)] If $d$ is odd, then $f_{-\alpha(d)}$ has a parabolic fixed point.
    \item[(iv)] If $d$ is even and $\gamma(d)<c<\alpha(d)$, then $f_{c}$ has an attractive fixed~point. 
    \item[(v)] If $d$ is even, then $f_{\gamma(d)}$ has a parabolic fixed point.
    \item[(vi)] If $d$ is even and $-1 < c < \gamma(d)$, then $f_c$ has an attractive cycle of period~$2$.
\end{itemize}
\end{lemma}

We think this is a special case of the general phenomenon called \textit{period-doubling bifurcation}, but we cannot find a good reference, so we provide an \textit{ad hoc} proof.

\begin{proof} Let $\delta(d)=d^{-1/(d-1)}$. We can easily check that

(i): the parabolic fixed point is $\delta(d)$;

(iii, v): the parabolic fixed point is $-\delta(d)$;

(ii, iv): the attractive fixed point is between $-\delta(d)$ and $\delta(d)$, by the intermediate value theorem.

For (vi), let $z$ be a period $2$ point of $f_c$ but not a fixed point. Then $z$ satisfies
\[
    \frac{f_c^2 (z)-z}{f_c (z)-z}=\frac{(z^d+c)^d-z^d+z^d+c-z}{z^d+c-z}=0,
\]
namely,
\[
    \sum_{k=0}^{d-1} z^k(z^d+c)^{d-1-k}+1=0.
\]
Let
\begin{equation}\label{eq:2.1}
    x=\frac{z^d+c}{z}.
\end{equation}
Then we have
\begin{equation}\label{eq:2.2}
     \sum_{k=0}^{d-1} x^k =-z^{-(d-1)}.
\end{equation}
Let $\lambda$ be the multiplier of $z$. Then we have
\[
    (f_c^2)'(z)=d^2z^{d-1}(z^d+c)^{d-1}=\lambda,
\]
and
\begin{equation}\label{eq:2.3}
    \lambda^{-1}d^2x^{d-1}= z^{-2(d-1)}.
\end{equation}
Combining \eqref{eq:2.2} and \eqref{eq:2.3}, we get
\begin{equation}\label{eq:2.4}
    g_\lambda(x):=\lambda\left(\sum_{k=0}^{d-1} x^k\right)^2-d^2x^{d-1}=0.
\end{equation}

Now suppose that $0<\lambda<1$. Then by Descartes' sign rule, the polynomial $g_\lambda$ has at most two positive real roots.
Since $g_\lambda(0)=\lambda>0$, $g_\lambda(1)=d^2(\lambda-1)<0$, and $g_\lambda(+\infty)=+\infty$, we know that $g_\lambda$ has a unique root $0<x<1$.

We claim that $\lambda\mapsto x$ is a bijection from $(0,1)$ to $(0,1)$. It is clearly injective by \eqref{eq:2.4}. It is also surjective because if $0<x<1$, then 
\[
    0<\frac{d^2x^{d-1}}{\left(\sum_{k=0}^{d-1} x^k\right)^2}<\frac{d^2x^{d-1}}{d^2\left(\prod_{k=0}^{d-1}x^k\right)^{2/d}}=1.
\]
Combining \eqref{eq:2.1} and \eqref{eq:2.2}, we get
\[
    c = z(x-z^{d-1})=\frac{-\sum_{k=0}^{d}x^k}{\left(\sum_{k=0}^{d-1}x^k\right)^{d/(d-1)}}.
\]
If $x=0$, then $c=-1$. If $x=1$, then $c=\gamma(d)$.
Since the map $x\mapsto c$ is continuous, any 
$-1 < c < \gamma(d)$ is the image of some $0<x<1$, which is in turn the image of some 
$0<\lambda<1$.
\end{proof}

The following lemma is a special case of {\cite[Theorems III.2.2 and III.2.3]{MR1230383}}.

\begin{lemma}\label{lemma: carleson-gamelin}
Any attractive or parabolic cycle of $f_c$ attracts the orbit of the critical point $0$. In particular:
\begin{itemize}
    \item[(i)] The polynomial $f_c$ has at most one attractive or parabolic cycle.
    \item[(ii)] If $f_c$ has an attractive or parabolic fixed point $z_c$, then the orbit of the critical point $0$ converges to $z_c$.
\end{itemize}
\end{lemma}

\subsection{Arithmetic of parabolic parameters}
Milnor studied the arithmetic of unicritical polynomials and proved the following property of $\Par_d$.
\begin{lemma}[{\cite[Theorem 1.1 and Remark 2.2]{MR3289903}}]\label{lemma: Milnor}
    If $c \in \Par_d$, then
$$ d^{\frac{d}{d-1}}c\in\overline\Z \qquad and \qquad \left(d^{\frac{d}{d-1}}c\right)\overline \Z + d \overline\Z = \overline\Z.$$    
\end{lemma}

This lemma gives a strong restriction on the elements of $\Par_d$.

For any prime number $p$, there is a $p$-adic valuation $\nu_p:\Q\to\Z\cup\{\infty\}$ with $\nu_p(p) = 1$. It can be extended to a valuation map $\nu_p:\overline\Q\to\Q\cup\{\infty\}$. Note that this extension is not unique. We fix one such extension.

\begin{corollary}\label{cor: p-adic valuation of c}
If $c \in \Par_d$, then $\nu_p(c) = -\nu_p \left(d^{\frac{d}{d-1}} \right)$ if $p$ divides $d$ and $\nu_p(c) \ge 0$ otherwise.
\end{corollary}
\begin{proof}
By Lemma \ref{lemma: Milnor}, there are $x, y \in \overline\Z$ such that $d^{\frac{d}{d-1}}c x + d y = 1$.
If $p$ divides~$d$, then $\nu_p\left(d^{\frac{d}{d-1}}c\right) = 0$. Indeed, if $\nu_p\left(d^{\frac{d}{d-1}}c\right) > 0$, then $0 < \nu_p\left(d^{\frac{d}{d-1}}c x + d y\right) = \nu_p(1) = 0$, which is a contradiction.
If $p$ does not divide~$d$, then $0 \le \nu_p\left(d^{\frac{d}{d-1}}c\right) = \nu_p\left(d^{\frac{d}{d-1}}\right)+\nu_p(c) = \nu_p(c)$.
\end{proof}

\begin{corollary}\label{cor: leading coeff of P}
Let $c\in\Par_d$ and $P(T)=a_nT^n+\dots+a_0\in\Z[T]$ be the minimal polynomial of $c$ with $a_n>0$. Then we have $a_n=d^{\frac{nd}{d-1}} \in \Z$.
\end{corollary}
\begin{proof}
Let $c_{1}=c,c_{2},\dots,c_{n}$ be the Galois conjugates of $c$. They are also parabolic. For any $Q \in \overline\Q[T]$, we denote by $\nu_p(Q)$ the minimum of the $p$-adic valuation of its coefficients. By Gauss's lemma~\cite[Lemma 1.6.3]{MR2216774}, it follows that for any prime $p$, we have
\[
    0 
    = \nu_p(P) 
    = \nu_p(a_n) + \sum_{i = 1}^n \nu_p(T - c_i) 
    =\nu_p(a_n) + \sum_{i = 1}^n \min\{0, \nu_p(c_i)\}.
\]
Thus, Corollary~\ref{cor: p-adic valuation of c} implies that $\nu_p(a_n) = \frac{nd}{d-1}\nu_p(d)$ if $p$ divides $d$ and $\nu_p(a_n) = 0$ otherwise. Hence, we have $a_n = d^{\frac{nd}{d-1}}\in \Z$.
\end{proof}

\subsection{Capacity theory} We recall a lemma from capacity theory, which was used by Noytaptim and Petsche.

Let $K \subseteq \C$ be a compact set. For any integer $n \ge 2$, we define the \textit{n-th diameter } of $K$ as
\[
    d_n(K) = \sup_{z_1, \dots, z_n \in K}
    \prod_{1 \le i < j \le n} |z_i - z_j|^{\frac{2}{n(n - 1)}}.
\]
The sequence $\{d_n(K) \colon n \ge 2\}$ is monotone decreasing, by \cite[Theorem~5.5.2]{MR1334766}. 
The \textit{transfinite diameter} of $K$ is defined as
\[
    d_\infty(K) = \lim_{n \rightarrow \infty} d_n(K).
\]
The transfinite diameter of a real interval $[a, b] \subseteq \R$ is $d_\infty([a,b])=(b - a)/4$, see \cite[Corollary 5.2.4]{MR1334766}. The $n$-th diameters of $[a, b]$ were explicitly computed in \cite{MR4710201}, but it seems to us that they were already known by Stieltjes~\cite{stieltjes1885quelques} and Schur~\cite{MR1544303}.

\begin{lemma}[{\cite[Theorem~2]{MR4710201}}]\label{lemma: nth diameter}
For $n \ge 2$, the $n$-th diameter of a real interval $[a, b]$ is given by
\[
    d_n([a, b]) = (b - a) D_n^{\frac{1}{n(n - 1)}}
\]
where $\{D_n \colon n \ge 2\}$ is defined recursively by $D_2 = 1$ and 
\[
    D_n = \frac{n^n(n - 2)^{n - 2}}{2^{2n - 2}(2n - 3)^{2n - 3}} D_{n - 1}.
\]
\end{lemma}

Actually, we do not need the full lemma. We only use the fact that $D_n^{\frac{1}{n(n - 1)}}$ is decreasing and the first few values: $D_2=1$, $D_3=2^{-4}$, and $D_4=5^{-5}$.
\section{Proof of Theorem~\ref{theorem: noytaptim - petsche}}\label{section:3}

In this section, we provide an alternative proof of~Theorem~\ref{theorem: noytaptim - petsche}, following the idea of Buff and Koch.

The following lemma is extracted from their proof.

\begin{lemma}\label{lemma: galois orbit in -2, 0}
If $c \in \overline \Z$ and all its Galois conjugates are in the interval $[-2, 0]$, then $c \in \{0, -1, -2\}$.
\end{lemma}
\begin{proof}
Let $F(z) = z + z^{-1}$. Then $F$ maps the unit circle onto the interval $[-2, 2]$. In particular, $F$ maps the arc $A = \{e^{i \theta} \colon \frac{\pi}{2} \le \theta \le \frac{3\pi}{2}\}$ onto the interval $[-2, 0]$.

Let $c$ satisfy the given condition and $a\in F^{-1}(c)$. Then $a\in \overline \Z$ and all its Galois conjugates are in $A$. By Kronecker's theorem~\cite[Theorem 1.5.9]{MR2216774}, $a$ must be a root of unity. The condition $a\in A$ forces $a$ to be $-1$, $\pm i$, or $-\frac{1}{2} \pm \frac{\sqrt{3}}{2}i$. Hence, $c \in \{ 0, -1, -2 \}$.
\end{proof}

Now we are ready to give the proof of Theorem~\ref{theorem: noytaptim - petsche}.

\begin{proof}[Proof of Theorem~\ref{theorem: noytaptim - petsche}]
The inclusion $\supseteq$ is clear. Now let $c \in \PCF_d \cap \Qtr$. First we notice that by the definition, $c\in \overline\Z$ and all Galois conjugates of $c$ are PCF.  We distinguish two cases.

\textsc{Case 1}: $d$ is odd. We have $-\alpha(d) \le c \le \alpha(d)$. In particular, by Lemmas \ref{lemma: lemma 2.1} and ~\ref{lemma: carleson-gamelin}, $f_c$ has a fixed point $z_c$ with $f^n_c(0) \rightarrow z_c$ as $n\to\infty$. Given that the orbit $\{ f^n_c(0) \colon n \in \N\}$ is finite, we have $f^n_c(0) = z_c$ for large enough $n$. Since $f_c$ is strictly increasing on $\R$, it follows that $z_c = 0$ and $c = 0$.

\textsc{Case 2}: $d$ is even. We first assume that $0 \le c \le \alpha(d)$. As in \textsc{Case 1}, by Lemmas \ref{lemma: lemma 2.1} and ~\ref{lemma: carleson-gamelin},  $f_c$ has a fixed point $z_c$ with $f^n_c(0) \rightarrow z_c$ as $n\to\infty$. Given that the orbit $\{ f^n_c(0) \colon n \in \N\}$ is finite, we have $f^n_c(0) = z_c$ for large enough $n$. Again, since $f_c$ is strictly increasing on $\R_{\ge 0}$, we have $z_c = 0$ and $c = 0$.
This argument works for all Galois conjugates of $c$. Therefore, we can assume that all Galois conjugates of $c$ are contained in $[\beta(d), 0] \subseteq [-2, 0]$. By Lemma~\ref{lemma: galois orbit in -2, 0}, we have $c \in \{0, -1, -2\}$ if $d=2$ and $c \in \{0, -1\}$ if $d\ge 4$.
\end{proof}
\section{Proof of Theorem~\ref{theorem: theorem}}\label{section:4}

In this section, we give the proof of Theorem~\ref{theorem: theorem}. From now on, we assume that $c\in\Par_d\cap\Qtr$. 

Let $P(T)=a_nT^n+\dots+a_0\in\Z[T]$ be the minimal polynomial of $c$ with $a_n>0$, let
$c_{1}=c,c_{2},\dots,c_{n}$ be the Galois conjugates of $c$, and let
\[
    \Delta(P) = a_n^{2(n - 1)} \prod_{1 \le i < j \le n} (c_i - c_j)^2 \in \Z
\]
be the discriminant of $P$. Since $c$ is totally real, we have $\Delta(P)>0$.

We first bound $\Delta(P)$ from below and above.

\begin{lemma}\label{lemma: bound discriminant}
Assume that $c \in \Par_d \cap \Qtr$ and $d\ge 4$ is even. Then $a_n^{n - 1}$ divides $\Delta(P)$ and
\[
    a_n^{n-1} 
    \le \Delta(P) 
    \le a_n^{2(n - 1)} \left(2^{\frac{1}{d-1}} - 1\right)^{n(n - 1)} D_n,
\]
where $D_n$ is defined in Lemma \ref{lemma: nth diameter}.
\end{lemma}
\begin{proof}
For any prime $p$ that divides $d$, Corollaries~\ref{cor: p-adic valuation of c} and \ref{cor: leading coeff of P} imply that
\begin{align*}\label{eq:2}
    \nu_{p}(\Delta(P)) 
    & = 2(n-1)\nu_{p}(a_{n}) + 2 \sum_{1 \le i < j \le n} \nu_{p} (c_i - c_j)\\
    & \geq 2 (n-1) \frac{nd}{d-1} \nu_{p}(d) - n(n-1) \frac{d}{d-1} \nu_{p}(d)\\
    & = n(n-1) \frac{d}{d-1} \nu_{p}(d)\\
    & = (n-1) \nu_{p}(a_{n}).
\end{align*}
Hence, it follows that $a_n^{n-1}$ divides $\Delta(P)$ and that $a_n^{n-1} \le \Delta(P)$.

Suppose that $d\ge 4$ is even. If $c \in \Par_d$ and $c > -1$, then $c = \alpha(d)$ or $c = \gamma(d)$ by Lemmas \ref{lemma: lemma 2.1} and \ref{lemma: carleson-gamelin}. However, neither of them is totally real. Therefore, we can assume that 
$c_1,\dots,c_n \in [\beta(d), -1]$ and get the upper bound
\[
    \Delta(P)
    \leq a_{n}^{2(n-1)}d_{n}([\beta(d),-1])^{n(n-1)}.
\]
By Lemma~\ref{lemma: nth diameter}, we have 
\[
    \Delta(P) 
    \le a_n^{2(n-1)} \left(( - 1-\beta(d))D_{n}^{\frac{1}{n(n-1)}}\right)^{n(n-1)}
    =a_n^{2(n - 1)} \left(2^{\frac{1}{d-1}} - 1\right)^{n(n - 1)} D_n.
\]
This completes the proof.
\end{proof}

Now we are ready to give the proof of Theorem~\ref{theorem: theorem}.

\begin{proof}[Proof of Theorem~\ref{theorem: theorem}]
If $d$ is odd, then  the only real parabolic points are $\pm \alpha(d)$ by Lemmas \ref{lemma: lemma 2.1} and \ref{lemma: carleson-gamelin}. They are totally real only if $d=3$.
    
If $d \ge 4$ is even, then by Corollary \ref{cor: leading coeff of P} and Lemma~\ref{lemma: bound discriminant}, we have
\begin{equation}\label{eq:4.1}
    d^{\frac{d}{d-1}} \left(2^{\frac{1}{d-1}} - 1\right) D_{n}^{\frac{1}{n(n-1)}}\geq1.
\end{equation}
We can check that $\sigma(d) := d^{\frac{d}{d-1}}\left(2^{\frac{1}{d-1}}-1\right)$
is decreasing and converges to $\log 2$ as $d \rightarrow \infty$. Moreover, we have seen that $\tau(n) := D_{n}^{\frac{1}{n(n-1)}}$ is decreasing and converges to $1/4$ as $n \rightarrow \infty$.

We want to show by contradiction that $\Par_d \cap \Qtr  = \emptyset$ if $d \ge 4$ is even. We suppose to the contrary and distinguish four subcases.

\textsc{Case 1}: If $n \in \{1, 2\}$ and $d \geq 4$, then $a_{n}=d^{\frac{nd}{d-1}}$ is not an integer, contradicting Corollary~\ref{cor: leading coeff of P}.

\textsc{Case 2}: If $n\geq3$ and $d\geq6$, then
\[
    \sigma(d)\tau(n)\leq\sigma(6)\tau(3)=6^{\frac{6}{5}}(2^{\frac{1}{5}}-1)2^{-\frac{2}{3}}<1,
\]
contradicting \eqref{eq:4.1}.

\textsc{Case 3}: If $n\geq4$ and $d\geq4$, then
\[
    \sigma(d)\tau(n)\leq\sigma(4)\tau(4)=4^{\frac{4}{3}}(2^{\frac{1}{3}}-1)5^{-\frac{5}{12}}<1,
\]
contradicting \eqref{eq:4.1}.

\textsc{Case 4}: If $n=3$ and $d=4$, then
\begin{align*}
    P(T) & =a_{3}(T - c_1)(T - c_2)(T - c_3)\\
    & = a_{3} T^{3} + a_{2} T^{2} + a_{1}T + a_{0},\text{ where }a_{3} = 2^{8},
\end{align*}
and 
\[
    \Delta(P) 
    = a_{1}^{2}a_{2}^{2} 
    - 4 a_{1}^{3}a_{3} 
    - 4 a_{0}a_{2}^{3}
    - 27 a_{0}^{2}a_{3}^{2} 
    + 18 a_{0}a_{1}a_{2}a_{3}.
\]
By Lemma~\ref{lemma: bound discriminant}, $\Delta(P)$ is a multiple of $a_{3}^{2}=2^{16}$ and
\[
    \Delta(P) 
    \leq a_{n}^{2(n-1)} \left(2^{\frac{1}{d-1}} - 1\right)^{n(n-1)}D_{n}=2^{28}(2^{\frac{1}{3}}-1)^{6}<2^{17}.
\]
Therefore, we must have $\Delta(P)=2^{16}$. Moreover, by Corollary \ref{cor: p-adic valuation of c}, we have
\begin{align}
\nu_{2}(a_{0}) & = \nu_{2}(2^{8}c_{1}c_{2}c_{3})=0,\nonumber \\
\nu_{2}(a_{1}) & = \nu_{2}(2^{8}(c_{1}c_{2}+c_{2}c_{3}+c_{3}c_{1}))\geq\left\lceil 8-{\textstyle \frac{16}{3}}\right\rceil {\textstyle =3},\label{eq:4.2}\\
\nu_{2}(a_{2}) & = \nu_{2}(2^{8}(c_{1}+c_{2}+c_{3}))\geq\left\lceil 8 - {\textstyle \frac{8}{3}}\right\rceil {\textstyle =6},\label{eq:4.3}\\
\nu_{2}(a_{3}) & = \nu_{2}(2^{8})=8.\nonumber 
\end{align}
Then
\begin{align*}
\nu_{2}(a_{1}^{2}a_{2}^{2}) & \geq3\cdot2+6\cdot2=18,\\
\nu_{2}(-4a_{1}^{3}a_{3}) & \geq2+3\cdot3+8=19,\\
\nu_{2}(-4a_{0}a_{2}^{3}) & \geq2+0+6\cdot3=20,\\
\nu_{2}(18a_{0}a_{1}a_{2}a_{3}) & \geq1+0+3+6+8=18.
\end{align*}
If either \eqref{eq:4.2} or \eqref{eq:4.3} is strict, then $\nu_{2}(a_{1}^{2}a_{2}^{2}), \nu_{2}(18a_{0}a_{1}a_{2}a_{3})\geq19$.
Otherwise, $\nu_{2}(a_{1}^{2}a_{2}^{2}),\nu_{2}(18a_{0}a_{1}a_{2}a_{3})=18$
and $\nu_{2}(a_{1}^{2}a_{2}^{2}+18a_{0}a_{1}a_{2}a_{3})\geq19$. As
a result,
\[
    2^{16}\equiv\Delta(P)\equiv-27a_{0}^{2}a_{3}^{2}\equiv5\cdot2^{16}\text{ mod }2^{19}.
\]
This is a contradiction, so such $P$ does not exist.
\end{proof}

\defbibheading{mybibheading}{
  \section*{References}
}
\printbibliography[heading = mybibheading]

\end{document}